\documentclass[11pt, reqno]{amsart}
\usepackage[dvips]{graphicx}
\usepackage{bmpsize} 
\usepackage[margin=1.27in]{geometry}
\usepackage{mathrsfs}
\usepackage{amsmath,amscd}
\usepackage{amssymb}
\usepackage{amscd}
\usepackage{xcolor}
\usepackage{enumitem}
\usepackage{tikz-cd}
\usepackage{mathtools}
\graphicspath{/home/jhois/Documents/Images/}

\newtheorem{theorem}[equation]{Theorem}
\newtheorem*{theorem*}{Theorem}

\newtheorem*{lemma*}{Lemma}
\newtheorem{proposition}[equation]{Proposition}

\newtheorem*{corollary*}{Corollary}

\theoremstyle{remark}
\newtheorem{remark}[equation]{Remark}

\numberwithin{equation}{section} 

\newcommand{\R}{\mathbb R}

\begin{document}

\author[J.A.~Hoisington]{Joseph Ansel Hoisington}\address{Department of Mathematics and Statistics, Smith College, Northampton, MA 01063 USA}\address{(Current Address: Department of Mathematics, University of Georgia, Athens, GA 30602 USA)}\email{jhoisington@uga.edu} 

\title[Steiner's Formula and the Isoperimetric Inequality]{Steiner's Formula and a Variational Proof of the Isoperimetric Inequality}


\begin{abstract}
We give a new proof of the isoperimetric inequality in the plane, based on Steiner's formula for the area of a convex neighborhood.  This proof establishes the isoperimetric inequality directly, without requiring that we separately establish the existence of an optimal domain.  In doing so, this proof bypasses the main difficulty in all of the proofs Steiner outlined for the plane isoperimetric inequality.  
\end{abstract}

\maketitle



\section{Introduction}
\label{intro} 

The classical isoperimetric inequality states that among all simple closed curves of length $L$ in the plane, the unique curve enclosing the largest area is the circle of circumference $L$: 
\smallskip

\begin{theorem}[The Isoperimetric Inequality] 
\label{ie}

Let $\gamma$ be a simple closed curve in the plane of length $L$, enclosing a domain $D$ of area $A$. 
\medskip 

Then $L^{2} \geq 4\pi A$, with equality precisely if $\gamma$ is a circle.  

\end{theorem}
\smallskip

This paper gives a proof of the isoperimetric inequality based on Steiner's formula, which describes the area of a neighborhood of a convex domain in $\R^{2}$: 
\smallskip

\begin{theorem}[Steiner's Formula, \cite{St}] 
\label{sf}

Let $D$ be a bounded, convex domain in $\R^{2}$, of area $A$ and perimeter $L$, and let $D_{r}$ be the $r$-neighborhood of $D$, i.e. the points in $\R^{2}$ whose distance from $D$ is $r$ or less.  Then:  
\smallskip
\begin{itemize}
\item[{\bf A.}] $Area(D_{r}) = \pi r^{2} + L r + A$, 
\medskip
\item[{\bf B.}] $Length(\partial D_{r}) = 2\pi r + L$.
\end{itemize}
\end{theorem}  
\smallskip

Jakob Steiner (March $18^{th}$, 1796 - April $1^{st}$, 1863) proved Theorem \ref{sf} for convex polygons and a similar formula for convex polyhedra in $\R^{3}$.  By polygonal approximation, Theorem \ref{sf} then follows for any compact, convex set in $\R^{2}$, and in fact a version of Theorem \ref{sf} holds in much greater generality -- for more about Steiner's formula, see \cite{Gr,Sc}.  Steiner was fascinated by the isoperimetric inequality, and he sketched several ideas for proving it -- cf. \cite{Tr,Bl}.  The isoperimetric problem was already ancient when Steiner considered it in the nineteenth century, but Theorem \ref{ie} had never been proven rigorously.  It remained unproven in Steiner's lifetime, and all of Steiner's ideas for proving the isoperimetric inequality required the same additional step, which he never provided:  one must show that the isoperimetric problem has a solution. \\  

More precisely, we define the {\bf isoperimetric ratio} of a domain $D$ with area $A$ and perimeter $L$ to be: 

\begin{equation}
\label{ir}
\displaystyle {\Huge \frac{L^{2}}{4\pi A}}.
\end{equation}
\medskip

The isoperimetric ratio is scale-invariant -- we formulate the isoperimetric inequality in terms of $L^{2}$ and $A$, as in Theorem \ref{ie}, because $L^{2}$ and $A$ transform the same under rescalings.  The isoperimetric inequality then states that the isoperimetric ratio of any plane domain is greater than or equal to $1$, with equality precisely for disks.  Steiner developed many proofs that no domain other than a disk could minimize the isoperimetric ratio, but he didn't establish the existence of a domain that minimizes (\ref{ir}). \\ 

The first proof of the existence of a domain minimizing the isoperimetric ratio seems to have been in unpublished lecture notes of Weierstrass in 1879, cf. \cite{Bl}.  The existence of an optimal isoperimetric domain in the plane is now known to be a consequence of several compactness theorems in metric geometry and geometric measure theory, however the proof below does not require that we establish the existence of a minimizer for the isoperimetric ratio -- we show directly that no domain can have an isoperimetric ratio less than $1$.  We believe part of the significance of our proof is that it shows how one of Steiner's ideas from convex geometry can be used to prove the isoperimetric inequality without separately establishing the existence of an optimal domain. \\ 

The basic observation for our proof is the following:  if $D$ is a bounded convex domain in $\R^{2}$, we can use Theorem \ref{sf} to calculate the isoperimetric ratio $\mathcal{I}(r)$ of the $r$-neighborhood of $D$ as a function of $r$.  Letting $A$ be the area of $D$ and $L$ its perimeter, we have: 

\begin{equation}
\label{sfir}
\displaystyle \mathcal{I}(r) = \frac{\left( 2\pi r + L \right)^{2}}{4\pi \left( \pi r^{2} + Lr + A \right)} = \frac{4\pi^{2} r^{2} + 4\pi L r + L^{2}}{4\pi^{2} r^{2} + 4\pi L r + 4\pi A} . 
\end{equation}
\medskip

Differentiating with respect to $r$, we have: 

\begin{equation}
\label{sfird}
\displaystyle \mathcal{I}'(r) = \frac{\left( 4\pi A - L^{2} \right) \left( 8\pi^{2} r + 4\pi L \right)}{\left( 4\pi^{2} r^{2} + 4\pi L r + 4\pi A \right)^{2}} = \frac{\left( 4\pi A - L^{2} \right) \left( \pi r + L \right)}{4\pi \left( \pi r^{2} + L r + A \right)^{2}}. 
\end{equation}
\medskip  

This implies that $\mathcal{I}(r)$ is a monotone function of $r$, decreasing if the isoperimetric ratio of $D$ is greater than $1$ and constant if the isoperimetric ratio of $D$ is equal to $1$.  If $D$ were a convex domain with an isoperimetric ratio less than $1$, $\mathcal{I}(r)$ would increase monotonically to $1$, the isoperimetric ratio of the disk, as $r$ goes to infinity.  As $r$ goes to infinity, the $r$-neighborhoods of any convex domain $D$, when rescaled to have constant area, converge to a disk -- see Proposition \ref{goodvar}.  We will see that this gives a variation of the disk, as an argument for the functional on plane domains given by the isoperimetric ratio.  We will use Steiner's formula to find its first and second variations -- in particular, we will relate them to the isoperimetric ratio of the domain $D$ in question.  We will then be able to deduce Theorem \ref{ie} from the fact that the disk is a critical point,  with non-negative second variation, for the isoperimetric ratio on plane domains.  For later reference, the quantity $L^{2} - 4\pi A$ whose negative appears in (\ref{sfird}) is called the {\bf isoperimetric deficit} of a domain. \\ 

It will be important in our proof that, in the plane, the convex hull $conv(D)$ of a non-convex domain $D$ always has a smaller isoperimetric ratio than $D$ itself: $conv(D)$ encloses a larger area than $D$ with a smaller perimeter.  Therefore, to prove Theorem \ref{ie}, it is enough to show that the isoperimetric inequality holds for convex domains.  Steiner was aware of this fact and used it in several of his ideas for proving the isoperimetric inequality.  In dimensions greater than $2$, this is no longer true: the isoperimetric ratio of a $3$-dimensional domain with volume $V$ and surface area $A$ is defined to be $\frac{A^{3}}{36 \pi V^{2}}$.  Like (\ref{ir}) for plane domains, the isoperimetric ratio of a domain in $\R^{3}$ is scale-invariant and the ball has isoperimetric ratio equal to $1$.  The isoperimetric inequality in $\R^{3}$ states that the isoperimetric ratio of any domain is greater than or equal to $1$, with the ball being the unique minimizer.  For a ball with a long spike in $\R^{3}$, both the volume and surface area, and thus the isoperimetric ratio, can be made arbitrarily close to that of the ball by making the spike narrow enough.  On the other hand, the convex hull of such a domain will be approximately a cone with a hemispherical cap, with an isoperimetric ratio significantly greater than $1$:  for a spike of length $\eta$ on the unit ball, the isoperimetric ratio of its convex hull will be approximately $\frac{\eta + 3}{4}$ for $\eta$ very large. \\ 

The outline of this paper and our proof of the isoperimetric inequality is as follows: \\ 

In Section \ref{variations}, we will calculate the first and second variations of the isoperimetric ratio of the disk.  We will show that the disk is a stable critical point of the isoperimetric ratio and that any variation has positive second variation unless, to first order, the variation is the sum of a translation and a rescaling of the disk. \\ 

In Section \ref{monotonicity}, we will use the $r$-neighborhoods of a compact, convex domain $D$ in the plane to construct a variation of the disk of the type analyzed in Section \ref{variations}.  We will use Steiner's formula to relate its first and second variations to the isoperimetric deficit of $D$, and in doing so, we will show that the isoperimetric deficit of $D$ is non-negative. \\ 

Once we know that the isoperimetric inequality $L^{2} - 4\pi A \geq 0$ holds, any of Steiner's arguments then prove that the disk is the only domain for which equality holds.  However, we will show in Section \ref{uniqueness} that the uniqueness of the disk as a minimizing domain also follows from our proof. \\  

We will prove that the perimeter $L$ and area $A$ of a plane domain $D$ satisfy $L^{2} \geq 4\pi A$ under the assumption that its boundary $\partial D$ is smooth, and we will make the further simplifying assumption that the curvature of $\partial D$ is strictly positive -- that is, the curvature vector of $\partial D$ always points into $D$ and never vanishes.  However by approximation (and the reduction to the convex case) this inequality then follows immediately for any plane domain with a rectifiable boundary.  The corresponding issue is more difficult in higher dimensions -- this is discussed in Section 2 of \cite{Os}.  In all dimensions, however, the boundary of a compact, convex domain can be realized as the Lipschitz image of a round sphere and is therefore rectifiable. \\ 

Throughout the paper, we will discuss the relationship between this proof and other known proofs of the isoperimetric inequality.  Robert Osserman's article \cite{Os} gives an overview of the isoperimetric inequality, its generalizations and their significance in mathematics.  Isaac Chavel's \cite{Ch} and Luis Santal\'o's \cite{San} books both discuss many results and questions in geometry and analysis which are based on the isoperimetric inequality and give several proofs of the classical isoperimetric inequality.  Bl\r{a}sj\"o discusses the history of the isoperimetric inequality in \cite{Bl}, and Howards, Hutchings and Morgan in \cite{HHM} and Andrejs Treibergs in \cite{Tr} present several proofs of the classical isoperimetric inequality. \\  

{\bf Acknowledgments:} I am very happy to thank Christopher Croke, Joseph H.G. Fu and Peter McGrath for their feedback about this work and Isaac Chavel, Frank Morgan and Franz Schuster for their input about the history of the isoperimetric inequality.  


\section{The First and Second Variations of the Isoperimetric Ratio}
\label{variations}

We will calculate the first and second variations of the isoperimetric ratio of the disk for variations through families of convex domains -- in particular, we will see that the disk is a critical point of the isoperimetric ratio and, infinitesimally, a minimizer. \\  

A compact, convex domain $D$ can be described by its {\bf support function} $p(\theta) : S^{1} \rightarrow \R$, defined as follows: 

\begin{align*}
\displaystyle p(\theta) = \text{max} \left(\lbrace h_{\theta}(x) := x_0 \cos(\theta) + x_1 \sin(\theta) \ | \ x = (x_0, x_1) \in D \rbrace \right).
\end{align*}
\smallskip 

If the boundary $\partial D$ of $D$ is smooth and has strictly positive curvature, then $p(\theta) + p''(\theta)$ is its radius of curvature.  In this case, the area $A$ and perimeter $l$ of $D$ are given by:  

\begin{equation}
\label{sup_area}
\displaystyle A = (\frac{1}{2})\int\limits_{0}^{2\pi} p(\theta) \left( p(\theta) + p''(\theta) \right) d\theta = (\frac{1}{2})\int\limits_{0}^{2\pi} p(\theta)^{2} - p'(\theta)^{2}  d\theta, 
\end{equation}
\begin{equation}
\label{sup_perim}
\displaystyle l = \int\limits_{0}^{2\pi} p(\theta)d\theta.  
\end{equation}
\smallskip 

This is described in Chapter 1 of \cite{San}.  A variation of the unit disk $\mathcal{D}_{0}$ through a family of such domains $\lbrace \mathcal{D}_{t} \rbrace_{t \geq 0}$ can therefore be described by a smooth function $p(\theta, t)$, with $p(\theta, t)$ the support function of the domain $\mathcal{D}_{t}$.  In particular, $p(\theta, 0) \equiv 1$. 

\begin{proposition}
\label{12var}

Let $\mathcal{D}_{t}$ be a family of compact, convex domains in the plane, with the boundary $\partial \mathcal{D}_{t}$ of each domain smooth and with positive curvature, which give a variation of the disk $\mathcal{D}_{0}$ as above.  Let $I(t)$ be the isoperimetric ratio of the domain $\mathcal{D}_{t}$. 
\medskip

Then $I'(0) = 0$ and $I''(0) \geq 0$, with equality if and only if, to first order, the family of domains coincides with a rescaling and translation of the disk.

\end{proposition}

\begin{proof} Let $p(\theta,t)$ be the support function of $\mathcal{D}_{t}$ as above.  Then letting $A(t)$ be the area and $l(t)$ the perimeter of $\mathcal{D}_{t}$, by (\ref{sup_area}) and (\ref{sup_perim}) we have: 
	
\begin{equation}
\label{var_area}
\displaystyle A(t) = (\frac{1}{2})\int\limits_{0}^{2\pi} p(\theta, t)^{2} - \frac{\partial p}{\partial \theta} (\theta, t)^{2}  d\theta, 
\end{equation}
\begin{equation}
\label{var_perim}
\displaystyle l(t) = \int\limits_{0}^{2\pi} p(\theta, t) d\theta. 
\end{equation}
\smallskip 

Because $p(\theta, 0) \equiv 1$ and $\frac{\partial p}{\partial \theta}(\theta, 0) \equiv 0$, $A'(0)$ and $l'(0)$ are both equal to $\int_{0}^{2\pi} \frac{\partial p}{\partial t}(\theta, 0) d\theta$. \\  

We then have that $I'(0) = \frac{2A(0)  l(0)  l'(0) - A'(0) l(0)^{2}}{4 \pi A(0)^{2}} $ is equal to:  

\begin{align*}
\displaystyle \frac{2 \times \pi \times 2\pi \left( \int\limits_{0}^{2\pi} \frac{\partial p}{\partial t}(\theta, 0) d\theta \right) - 2\pi \times 2\pi \left( \int\limits_{0}^{2\pi} \frac{\partial p}{\partial t}(\theta, 0) d\theta\right)}{4\pi^{3}} = 0. 
\end{align*} 
\smallskip  

$l''(0)$ is equal to $\int_{0}^{2\pi} \frac{\partial^{2} p}{\partial t^{2}}(\theta, 0) d\theta$ and, using again that $p(\theta, 0) \equiv 1$ and $\frac{\partial p}{\partial \theta}(\theta, 0) \equiv 0$, we have:  

\begin{equation}
\label{area_second_derivative}
\displaystyle A''(0) = \int\limits_{0}^{2\pi} \left[ \frac{\partial p}{\partial t}(\theta, 0)^{2} + \frac{\partial^{2} p}{\partial t^{2}}(\theta, 0) - \frac{\partial^{2} p}{\partial t \partial \theta}(\theta, 0)^{2} \right] d\theta.   
\end{equation}
\smallskip 

We then have that $I''(0) = \frac{ \left(2 A'(0) - l'(0) \right)^{2} + 2\pi \left(l''(0) - A''(0) \right)}{2 \pi^{2}}$ is equal to: 

\begin{equation}
\label{wi}
\displaystyle \frac{\left(\int\limits_{0}^{2\pi} \frac{\partial p}{\partial t}(\theta, 0) d\theta \right)^{2} + 2\pi \left(\int\limits_{0}^{2\pi} \frac{\partial^{2} p}{\partial \theta \partial t}(\theta, 0)^{2} - \frac{\partial p}{\partial t}(\theta, 0)^{2} d\theta \right)}{2 \pi^{2}}.  
\end{equation}
\smallskip 

{\bf Wirtinger's inequality} states that if $\varphi(\theta)$ is a $2\pi$-periodic, continuously differentiable function with $\int_{0}^{2\pi} \varphi(\theta) d\theta = 0$, then: 

\begin{align*}
\displaystyle \text{\Large $\int\limits_{0}^{2\pi}$}\varphi'(\theta)^{2} d\theta \geq \text{\Large $\int\limits_{0}^{2\pi}$} \varphi(\theta)^{2} d\theta. 
\end{align*}
\smallskip

Equality holds precisely if $\varphi(\theta) = a_{0}\cos(\theta) + a_{1}\sin(\theta)$ for some constants $a_{0}, a_{1}$.  Wirtinger's inequality thus implies by (\ref{wi}) that $I''(0) \geq 0$ and is strictly positive unless $\frac{\partial p}{\partial t}(\theta, 0) = a_{0}\cos(\theta) + a_{1}\sin(\theta) + 2\pi \widehat{p}$, where $\widehat{p} = \frac{1}{2\pi}\int_{0}^{2\pi}\frac{\partial p}{\partial t}(\theta, 0) d\theta$.  The variation corresponding to $a_{0}\cos(\theta) + a_{1}\sin(\theta)$ gives a translation of the disk, in the direction whose argument is $\arctan(\frac{a_{1}}{a_{0}})$ at speed $\sqrt{a_{0}^{2} + a_{1}^{2}}$, and the variation corresponding to $2\pi \widehat{p}$ rescales the disk, by a factor $1 + t_{0} 2\pi \widehat{p}$ when $t = t_{0}$. \end{proof} 

Wirtinger's inequality can be proved by comparing the Fourier series of a $2\pi$-periodic function with that of its derivative, cf. \cite{Fo}.  Wirtinger's inequality also implies the isoperimetric inequality directly.  This was discovered by Hurwitz, who gave the first proof of the isoperimetric inequality based on Fourier analysis and Wirtinger's inequality in \cite{Hu}.  A variant of this proof, in which the role of Wirtinger's inequality is made explicit, can be found in \cite{Os} and \cite{BG}.  As with our proof, Hurwitz's proof of the isoperimetric inequality does not require that one separately establish the existence of a minimizing domain -- his argument shows directly that $l^{2} \geq 4\pi A$ for any plane domain, with equality precisely when the domain is a disk. 


\section{Steiner's Formula and the Monotonicity of the Isoperimetric Ratio}
\label{monotonicity}

To prove Theorem \ref{ie}, we begin by confirming that the $r$-neighborhoods of a bounded, convex domain $D$, when rescaled to have constant area, give a variation of the disk of the type considered in Proposition \ref{12var}:  

\begin{proposition}  
\label{goodvar}

Let $D$ be a compact, convex domain in the plane whose boundary is smooth and has positive curvature.  For $t > 0$, let $\mathcal{D}_{t}$ be the $r = \frac{1}{t}$-neighborhood of $D$, rescaled to have the same area as $D$, and let $\mathcal{D}_{0}$ be a disk with the same area as $D$. 
\medskip

Then $\lbrace \mathcal{D}_{t} \rbrace_{t \geq 0}$ gives a variation of the disk $\mathcal{D}_{0}$, as in Proposition \ref{12var}.  More precisely, if $q(\theta)$ is the support function of $D$, this variation is described by:  

\begin{equation}
\label{var_function}
\displaystyle p(\theta, t) = \text{\footnotesize $\sqrt{\frac{A}{A t^{2} + lt + \pi}}$} \left(q(\theta) t + 1 \right), 
\end{equation} 
\smallskip 

where $A$ is the area and $l$ is the perimeter of $D$.  

\end{proposition}

\begin{proof} 
	
Let $D$ be as above -- without loss of generality, suppose $D$ has area $\pi$.  Note first that each $r$-neighborhood of $D$ is also convex, cf. Remark \ref{ms} below, so that the variation in question is through a family of convex sets.  If $q(\theta)$ is the support function of $D$, then $q(\theta) + r$ is the support function of $D_{r}$ and, by Theorem \ref{sf}, $\scriptstyle \sqrt{\frac{\pi}{\pi r^{2} + lr + \pi}}$$(q(\theta) + r)$ is the support function of the rescaling of $D_{r}$ whose area is equal to that of $D$.  Rewriting this in terms of $t = \frac{1}{r}$ for $r > 0$, we have:  

\begin{equation}
\label{good_var_eqn}
\displaystyle p(\theta, t) = \text{\footnotesize $\sqrt{\frac{\pi}{\pi (\frac{1}{t})^{2} + l\frac{1}{t} + \pi}}$} \left( q(\theta) + \frac{1}{t} \right) = \text{\footnotesize $\sqrt{\frac{\pi}{\pi t^{2} + lt + \pi}}$} \left(q(\theta) t + 1 \right).  
\end{equation}
\smallskip 

We then have:  

\begin{equation*}
\displaystyle p(\theta,t) + \frac{\partial^{2} p}{\partial \theta^{2}}(\theta,t) = \text{\footnotesize $\sqrt{\frac{\pi}{\pi t^{2} + lt + \pi}}$} \left(t(q(\theta) + q''(\theta)) + 1 \right).    
\end{equation*}
\smallskip 

Since the curvature of $\partial D$ is positive, $q(\theta) + q''(\theta) > 0$, so for all $t > 0$ we also have that $p(\theta,t) + \frac{\partial^{2} p}{\partial \theta^{2}}(\theta,t) > 0$, and that $\partial \mathcal{D}_{t}$ has positive curvature.  $p(\theta,t)$ extends smoothly to $t=0$, where it is equal to the support function of the unit disk, and gives a variation of the disk as in Proposition \ref{12var}. \end{proof} 

\begin{remark}
\label{ms}

The $r$-neighborhood $D_{r}$ of a compact, convex set $D$ is the {\bf Minkowski sum} of $D$ with a disk of radius $r$ in $\R^{2}$.  Minkowski summation of convex sets is discussed extensively in \cite{Sc} and many other texts on convex and integral geometry. 
\end{remark}

We now prove the inequality in Theorem \ref{ie} -- that for a compact domain in $\R^{2}$ with perimeter $l$ and area $A$, $l^{2} \geq 4\pi A$.  We will then address the characterization of the equality case in Section \ref{uniqueness}.  

\begin{proof}[Proof of Theorem \ref{ie}, Part 1]
	
Let $D$ be a compact, convex domain in the plane with area $A$ and boundary length $l$, and suppose $\partial D$ is smooth and has positive curvature as above.  By (\ref{sfir}), for $t > 0$, the isoperimetric ratio $I(t)$ of the $(\frac{1}{t})$-neighborhood of $D$ is:  

\begin{equation}
\label{sfir2}
\displaystyle I(t) = \frac{l^{2}t^{2} + 4\pi l t + 4\pi^{2}}{4\pi A t^{2} + 4\pi l t + 4\pi^{2}}.    
\end{equation}
\smallskip  

Letting $\delta$ be the least absolute value of the roots of $f(t) = 4\pi A t^{2} + 4\pi l t + 4\pi^{2}$, the denominator of (\ref{sfir2}) (see Remark \ref{roots} below), the function of $t$ defined by (\ref{sfir2}) extends smoothly to $(-\delta, \infty)$.  In particular, (\ref{sfir2}) extends smoothly to $t = 0$ to give the isoperimetric ratio of the variation $\lbrace \mathcal{D}_{t} \rbrace_{t \geq 0}$ of the disk described in Proposition \ref{goodvar}.  $I(t)$ is a monotone function of $t \geq 0$, with the sign of $I'(t)$ determined by the isoperimetric deficit of $D$: 

\begin{equation}
\label{sfird2}
\displaystyle I'(t) = \frac{\left( l^{2} - 4\pi A \right)\left( lt^{2} + 2\pi t \right)}{4\pi\left(At^{2} + lt + \pi\right)^{2}}.  
\end{equation}
\smallskip 

Therefore, $I'(0) = 0$ (which also follows from Propositions \ref{12var} and \ref{goodvar}) and for $t > 0$, $I'(t)$ has the same sign as the isoperimetric deficit of $D$.  To show that $l^{2} \geq 4\pi A$, we calculate the second derivative of $I(t)$:  

\begin{equation}
\displaystyle I''(t) = \left( \frac{l^{2} - 4\pi A}{2\pi} \right) \left( \frac{\pi^{2} - 3\pi A t^{2} - Alt^{3}}{(At^{2} + lt + \pi)^{3}} \right).
\end{equation}
\smallskip  

In particular, $\displaystyle I''(0) = \frac{l^{2} - 4\pi A}{2\pi^{2}}$.  The sign of $l^{2} - 4\pi A$ is the same as that of $I''(0)$, which by Proposition \ref{12var} is greater than or equal to $0$. \end{proof}
	
\begin{remark}
\label{roots}

The roots of the denominator of (\ref{sfir2}), $f(t) = 4\pi A t^{2} + 4\pi l t + 4\pi^{2}$, are: 

\begin{equation}
\displaystyle \frac{-l \pm \sqrt{l^{2} - 4\pi A}}{2A}.
\end{equation}
\smallskip 

The isoperimetric inequality is equivalent to the statement that the roots of this polynomial are real, and thus negative, and are distinct unless the domain in question is a disk.  For our purposes, it is enough simply to note that any real roots of $f(t)$ are negative since $f(t) \geq 4\pi^{2}$ when $t \geq 0$.  The roots of the Steiner polynomial were studied by Green and Osher in \cite{GO} (the Steiner polynomial of a domain with area $A$ and perimeter $l$ is $\pi r^{2} + lr + A$, with roots $\frac{-l \pm \sqrt{l^{2} - 4\pi A}}{2\pi}$).  They note that Steiner's formula implies the isoperimetric deficit of the $r$-neighborhood of $D$ is equal to that of $D$.    
\end{remark}


\section{The Uniqueness of the Disk}
\label{uniqueness} 

Once we have shown that $l^{2} \geq 4\pi A$ for all plane domains with perimeter $l$ and area $A$, and thus that the disk minimizes the isoperimetric ratio, any of Steiner's arguments then show that it is the unique minimizer.  The uniqueness of the disk as a minimizing domain for the isoperimetric ratio also follows from our argument, subject to some mild technical assumptions: 

\begin{proof}[Proof of Theorem \ref{ie}, Part 2] Let $D$ be a bounded domain in the plane with smooth (or $C^{2}$) boundary whose area $A$ and boundary length $l$ satisfy $l^{2} = 4\pi A$.  We can suppose $A = \pi$ and $l = 2\pi$.  Suppose in addition that the curvature of $\partial D$ is positive, as above.  By (\ref{sfir2}), in the variation $\lbrace \mathcal{D}_{t} \rbrace_{t \geq 0}$ of the disk constructed from $D$ as in Section \ref{monotonicity}, the isoperimetric ratio of $\mathcal{D}_{t}$ is equal to $1$ for all $t \geq 0$, and therefore $l(t) \equiv 2\pi$.  Therefore, 
	
\begin{equation}
\label{l_deriv}
\displaystyle l'(t) = \int\limits_{0}^{2\pi} \frac{\partial p}{\partial t}(\theta,t) d\theta \equiv 0, 
\end{equation}
\begin{equation}
\label{l_second_deriv}
\displaystyle l''(t) = \int\limits_{0}^{2\pi} \frac{\partial^{2} p}{\partial t^{2}}(\theta,t) d\theta \equiv 0.  
\end{equation}
\smallskip 

By (\ref{area_second_derivative}) and (\ref{l_second_deriv}), we then have:   

\begin{equation}
\label{a_second_deriv}
\displaystyle \int\limits_{0}^{2\pi} \left[ \frac{\partial p}{\partial t}(\theta, t)^{2} - \frac{\partial^{2} p}{\partial t \partial \theta}(\theta, t)^{2} \right] d\theta = A''(t) \equiv 0. 
\end{equation}
\smallskip 

By (\ref{l_deriv}), (\ref{a_second_deriv}) and Wirtinger's inequality, $\frac{\partial p}{\partial t}(\theta, t) = c_{0}(t) \cos(\theta) + c_{1}(t) \sin(\theta)$ for some functions $c_{0}(t), c_{1}(t)$ of $t$.  Letting $q(\theta)$ be the support function of $D$, by (\ref{var_function}), 

\begin{equation}
\displaystyle \frac{q(\theta) - 1}{(t + 1)^{2}} = c_{0}(t) \cos(\theta) + c_{1}(t) \sin(\theta).  
\end{equation}
\smallskip 

This then implies that $c_{0}(t) = \frac{d_{0}}{(t + 1)^{2}}$, $c_{1}(t) = \frac{d_{1}}{(t + 1)^{2}}$ for some constants $d_{0}, d_{1}$, and that $q(\theta) = d_{0} \cos(\theta) + d_{1} \sin(\theta) + 1$.  $D$ is therefore the unit disk centered at $(d_{0}, d_{1})$. \end{proof} 

We conclude with a few remarks about the technical assumptions in the proof of the characterization of equality above: \\ 

We have assumed the domain $D$ to be convex, and to have $C^{2}$ boundary whose curvature is strictly positive, so that it can be described by a $C^{2}$ support function $q(\theta)$.  However, by the reduction to the convex case, any domain realizing equality in the isoperimetric inequality must be convex.  Moreover, for any compact, convex set $D$ and $r >0$, the $r$-neighborhood $D_{r}$ of $D$ has $C^{1,1}$ boundary, which is therefore twice-differentiable almost everywhere.  If $D$ realizes equality in the isoperimetric inequality, then by (\ref{sfir}) each of its $r$-neighborhoods $D_{r}$ does as well, and by the convexity of $D_{r}$, the curvature of $\partial D_{r}$ is non-negative at all points where it is defined.  Thus, if one can show that a domain which realizes equality in the isoperimetric inequality, whose boundary is twice-differentiable almost everywhere, and has non-negative curvature at all points where its curvature is defined is a disk, one will have shown that $D_{r}$ is a disk for all $r > 0$, and thus that $D$ is a disk as well. \\  

The relationship between the regularity of the boundary of a domain and the regularity of its support function and the smoothness properties of $\partial D_{r}$ are both discussed in \cite{Sc}.  Osserman discusses the significance of the regularity assumed on the boundaries of domains in the isoperimetric inequality in Section 2 of \cite{Os}.  He notes that one can modify a smooth domain by adding ``wiggles" to its boundary, increasing its perimeter while leaving its area unchanged -- thus, ``one has the ironic situation that the more irregular the boundary, the stronger will be the isoperimetric inequality, but the harder it is to prove.  The fact is, the isoperimetric inequality holds in the greatest generality imaginable, but one needs suitable definitions even to state it."  


\end{document}